\documentclass[11pt,reqno]{amsart}
\usepackage[T1]{fontenc}
\usepackage{lmodern}
\usepackage{amsmath,amssymb,amsthm,mathtools,mathrsfs}
\usepackage[a4paper,margin=27mm]{geometry}
\usepackage{booktabs,needspace}
\usepackage{xcolor}

\usepackage[colorlinks=true,linkcolor=blue,citecolor=blue,urlcolor=blue]{hyperref}
\allowdisplaybreaks
\newtheorem{theorem}{Theorem}[section]
\newtheorem{proposition}[theorem]{Proposition}
\newtheorem{lemma}[theorem]{Lemma}
\newtheorem{corollary}[theorem]{Corollary}
\theoremstyle{definition}

\newtheorem{remark}[theorem]{Remark}
\numberwithin{equation}{section}

\title[Level two crank moments]{The Algebra of Parity-Twisted Crank Moments\\
and a Prime-Detecting Expression}
\author{Soon-Yi Kang}
\address{Department of Mathematics, Kangwon National University,
Chuncheon, Gangwon-do 24341, South Korea}
\email{sy2kang@kangwon.ac.kr}
\begin{document}

\begin{abstract}
We determine the algebra generated by the normalized parity-twisted crank moments and describe it explicitly as a proper subalgebra of the quasi-modular forms on $\Gamma_0(2)$ closed under $D=q\frac{d}{dq}$.
We derive congruences from differential identities and construct a prime-detecting expression involving only
the normalized twisted second moment, its products, and its derivatives.
Its coefficient of $q^n$ vanishes if and only if $n$ is prime, for every $n\ge2$.
We prove that its highest weight, eight, is minimal among prime-detecting elements of this algebra.
Finally, we identify the same algebra as the one generated by MacMahon's functions $B_k$, obtaining a corresponding prime-detecting expression in $B_1$.
\end{abstract}
\maketitle


\section{Introduction}

For a positive integer $k$, the Eisenstein series of weight $2k$ is defined by
\[
 E_{2k}(\tau)
 =1-\frac{4k}{B_{2k}}\sum_{n=1}^{\infty}\sigma_{2k-1}(n)q^n,
 \qquad q=e^{2\pi i\tau},
\]
where $\tau$ lies in the upper half-plane, $B_{2k}$ is the $2k$-th Bernoulli number, and
$\sigma_j(n)=\sum_{d\mid n}d^j$.
For $k\ge2$, the series $E_{2k}$ is a modular form of weight $2k$ on $\mathrm{SL}_2(\mathbb Z)$.
The weight $2$ Eisenstein series $E_2$, however, is not a modular form. Its transformation law contains an additional correction term, making it a quasi-modular form.

In fact, the algebra of quasi-modular forms on $\mathrm{SL}_2(\mathbb Z)$, denoted by
$\widetilde{\mathcal M}(\mathrm{SL}_2(\mathbb Z))$, is obtained by adjoining $E_2$ to the algebra of modular forms:
\[
 \widetilde{\mathcal M}(\mathrm{SL}_2(\mathbb Z))
 =\mathcal M(\mathrm{SL}_2(\mathbb Z))[E_2]
 =\mathbb C[E_2,E_4,E_6],
\]
where $\mathcal M(\mathrm{SL}_2(\mathbb Z))$ denotes the algebra of modular forms on 
$\mathrm{SL}_2(\mathbb Z)$.
A quasi-modular form of weight $k$ is therefore a weighted homogeneous polynomial of weight $k$ in $E_2,E_4,E_6$.  We also consider finite sums of forms of different weights, which we call mixed-weight quasi-modular forms.
A fundamental property of this algebra is its closure under the differential operator
\[
 D=\frac{1}{2\pi i}\frac{d}{d\tau}=q\frac{d}{dq},
\]
as follows from Ramanujan's differential identities:
\begin{equation}\label{eq:Ramanujan}
 DE_2=\frac{E_2^2-E_4}{12},\qquad
 DE_4=\frac{E_2E_4-E_6}{3},\qquad
 DE_6=\frac{E_2E_6-E_4^2}{2}.
\end{equation}

Quasi-modular forms also arise as generating functions in partition theory and enumerative geometry. Examples include generating functions counting covers of elliptic curves and orbifolds, and
curves on Abelian surfaces. We refer to Rhoades \cite[Section 1]{Rhoades} for examples and references.
More recently, Goujard--M\"oller \cite{GM} established quasi-modularity for generating functions counting torus covers with Siegel--Veech weights.

The arithmetic properties of quasi-modular forms provide another reason to study quasi-modular generating functions.
Quasi-modular forms on $\mathrm{SL}_2(\mathbb Z)$ of a fixed weight and with rational
Fourier coefficients are $p$-adic modular forms in the sense of Serre \cite{Serre}; see also \cite[Section~3.2]{GJT}.
For example,
\[
 E_2\equiv E_{p+1}\pmod p.
\]

This $p$-adic viewpoint provides a way to derive congruences for the coefficients of quasi-modular generating functions using the theory of modular forms.

Another arithmetic application concerns prime detection.
Leli\`evre's divisor-sum criteria lead to prime-detecting expressions involving Eisenstein series and their derivatives. This approach was used by Craig--van Ittersum--Ono \cite{CvIO} for MacMahon's
partition functions and by the author, together with Matsusaka and Shin \cite{KMS}, for their variants.

These arithmetic applications motivate the study of the algebras generated by families of quasi-modular forms.
Beyond the quasi-modularity of individual generating functions, one may ask which subalgebra of the full algebra of quasi-modular forms a family generates, how the generated algebra behaves under differentiation, and which arithmetic expressions it contains.

Another family of quasi-modular forms arises from the moments of a partition statistic called the crank.
Atkin and Garvan \cite{AG} established the quasi-modularity of normalized ordinary crank moment generating functions, and Rhoades \cite{Rhoades} extended this study to twisted crank moments.
In joint work with Kim and Lee \cite{KKL}, the author showed that the normalized ordinary crank moments generate $\mathbb C[E_2,E_4,E_6]$.

The present paper investigates the algebra generated by normalized crank moments twisted by the parity of the crank.
We ask which subalgebra of the quasi-modular forms on $\Gamma_0(2)$ these moments generate, whether it is closed under $D$, and whether it contains prime-detecting expressions.

We determine this algebra explicitly and show that it is a proper subalgebra of $\widetilde{\mathcal M}(\Gamma_0(2))$ closed under $D$.
Within this algebra, we construct a prime-detecting expression involving only the normalized twisted second
moment, its products, and its derivatives.
The expression detects all primes. 
Its highest weight is eight, which we prove is the smallest possible in this algebra.

We also identify the twisted crank algebra with the algebra generated by MacMahon's functions $B_k$.
As a consequence, our prime-detecting expression can be written in terms of $B_1$ and its derivatives, answering a question left open at the end of Section~4 of \cite{KMS}.

Section~\ref{sec:moments} introduces the twisted crank moments and
states the main results. Section~\ref{sec:structure} determines their
algebra and proves its closure under $D$.
Section~\ref{sec:congruences} establishes congruences for the normalized
moments, and Section~\ref{sec:primes} proves the prime-detection theorem
and the minimality of the highest weight.
Finally, Section~\ref{sec:comparison} identifies the twisted crank
algebra with the algebra generated by MacMahon's functions $B_k$
and relates their differential identities. It also states the resulting prime-detecting expression
in $B_1$ explicitly.


\section{Twisted crank moments and main results}
\label{sec:moments}

For a nonempty partition $\lambda$, let $\mu(\lambda)$ denote the number of ones in $\lambda$, and let $\nu(\lambda)$ denote the number of parts strictly larger than $\mu(\lambda)$.
The crank is defined by
\[
 \operatorname{crank}(\lambda):=
 \begin{cases}
  \text{the largest part of $\lambda$},
     &\text{if }\mu(\lambda)=0,\\
  \nu(\lambda)-\mu(\lambda),
     &\text{if }\mu(\lambda)>0.
 \end{cases}
\]
For $m\in\mathbb Z$ and $n\ge2$, let $M(m,n)$ denote the number of partitions of $n$ with crank $m$. With the conventions
\[
 M(0,0)=1,\qquad
 M(-1,1)=M(1,1)=1,\qquad M(0,1)=-1,
\]
and all remaining coefficients for $n=0,1$ equal to zero, the two-variable crank generating function is given by \cite{AG1988,Garvan}
\begin{equation}\label{eq:intro-crank-gf}
 C(\zeta;q)
 :=\sum_{n\ge0}\sum_{m\in\mathbb Z}M(m,n)\zeta^m q^n
 =\frac{(q)_\infty}
 {(\zeta q)_\infty(\zeta^{-1}q)_\infty},
\end{equation}
where
\[
 (a)_\infty:=(a;q)_\infty
 :=\prod_{j\ge0}(1-aq^j).
\]

For each nonnegative integer $r$, define
\[
 C_r(\zeta;q)
 :=\left(\zeta\frac{\partial}{\partial\zeta}\right)^r
 C(\zeta;q)
 =\sum_{n\ge0}\sum_{m\in\mathbb Z}
 m^rM(m,n)\zeta^m q^n.
\]
We normalize these functions by setting
\[
 \mathbf{C}_r(\zeta;q)
 :=\frac{C_r(\zeta;q)}{C(\zeta;q)}.
\]
At $\zeta=1$, these are the normalized ordinary crank moments; at $\zeta=-1$, they are the normalized
crank moments twisted by the parity of the crank.  Indeed,
\[
 C_r(-1;q)
 =\sum_{n\ge0}\left(\sum_{m\in\mathbb Z}
 (-1)^m m^rM(m,n)\right)q^n.
\]
The symmetry $M(-m,n)=M(m,n)$ implies that the odd moments vanish at both specializations. 
The product formulas
\[
 \frac{1}{C(1;q)}=(q;q)_\infty,
 \qquad
 \frac{1}{C(-1;q)}
 =\frac{(-q;q)_\infty^2}{(q;q)_\infty}
\]
show that both reciprocals have integer coefficients.
Since $C_{2k}(\pm1;q)\in\mathbb Z[[q]]$, we obtain
\[
 \mathbf C_{2k}(\pm1;q)\in\mathbb Z[[q]].
\]
For $k\ge1$, these series have zero constant term, since the empty partition contributes zero to every
positive-order moment.

Atkin and Garvan \cite[Theorem~4.2]{AG} proved that $\mathbf{C}_{2k}(1;q)$ is a mixed-weight quasi-modular form on $\mathrm{SL}_2(\mathbb Z)$ of highest weight at most $2k$. They used this structure to derive identities and congruences involving crank moments.
In joint work with Kim and Lee \cite{KKL}, the author showed that the normalized ordinary crank moments and the Eisenstein series generate the same $\mathbb Q$-algebra. 
Consequently,
\[
 \mathscr C^{(1)}:=\mathbb C[\mathbf{C}_{2k}(1;q):k\ge1]
 =\widetilde{\mathcal M}(\mathrm{SL}_2(\mathbb Z))
 =\mathbb C[E_2,E_4,E_6].
\]

For the specialization $\zeta=-1$, Rhoades \cite[Theorem~3.1]{Rhoades} proved that
$\mathbf{C}_{2k}(-1;q)$ is a mixed-weight quasi-modular form on $\Gamma_0(2)$ of highest weight at most $2k$.
Using this quasi-modularity, he proved that, for each positive even moment order, there are infinitely many
primes for which the twisted crank moments satisfy congruences along infinitely many non-nested
arithmetic progressions \cite[Corollary~3.2]{Rhoades}.

We consider the algebra
\[
  \mathscr C^{(2)}
 :=\mathbb C[\mathbf{C}_{2k}(-1;q):k\ge1]
 \subseteq\widetilde{\mathcal M}(\Gamma_0(2)).
\]
Let
\[
 E_{2,2}(\tau):=2E_2(2\tau)-E_2(\tau),
\]
which is a modular form of weight two on $\Gamma_0(2)$.
Since
\[
 \mathcal M(\Gamma_0(2))=\mathbb C[E_{2,2},E_4],
\]
the algebra of quasi-modular forms on $\Gamma_0(2)$ is
\[
 \widetilde{\mathcal M}(\Gamma_0(2))
 =\mathcal M(\Gamma_0(2))[E_2]
 =\mathbb C[E_2,E_{2,2},E_4].
\]

Our first theorem describes the algebra $\mathscr C^{(2)}$.

\begin{theorem}\label{thm:structure}
We have
\[
  \mathscr C^{(2)}
 =\mathbb C[E_2+2E_{2,2}]
 +(4E_{2,2}^2-E_4)\,
   \widetilde{\mathcal M}(\Gamma_0(2))
 \subsetneq\widetilde{\mathcal M}(\Gamma_0(2)),
\]
and
\[
 D(\mathscr C^{(2)})\subseteq\mathscr C^{(2)}.
\]
Moreover, adjoining the normalized ordinary second crank moment gives the full algebra:
\[
 \widetilde{\mathcal M}(\Gamma_0(2))
 =\mathbb C[
   \mathbf{C}_2(1;q),
   \mathbf{C}_2(-1;q),
   \mathbf{C}_4(-1;q)].
\]
\end{theorem}
As an arithmetic application of the differential identities in $\mathscr C^{(2)}$, we obtain the following congruences for the coefficients of the normalized twisted crank moments.

\begin{proposition}\label{prop:crank-congruences}
We have
\[
 \mathbf{C}_4(-1;q)
 \equiv(2D-1)\mathbf{C}_2(-1;q)\pmod5.
\]
Writing
\[
 \mathbf{C}_{2k}(-1;q)
 =\sum_{n\ge0}c_{2k}(n)q^n,
\]
we have, for every $j\ge1$ and $n\ge0$,
\[
 c_{4j}(5n+3)\equiv0\pmod5.
\]
\end{proposition}

We next turn to prime detection.
Craig, van Ittersum, and Ono \cite{CvIO} showed that the primes can be characterized as the solutions of equations involving MacMahon's partition functions.

For $a\ge1$, MacMahon's partition function $M_a(n)$ is defined by
\[
 \sum_{n\ge0}M_a(n)q^n
 :=\sum_{0<s_1<\cdots<s_a}
 \prod_{j=1}^a\frac{q^{s_j}}{(1-q^{s_j})^2}.
\]
Equivalently,
\[
 M_a(n)
 =\sum_{\substack{
     0<s_1<\cdots<s_a\\
     m_1,\ldots,m_a\ge1\\
     m_1s_1+\cdots+m_as_a=n}}
 m_1\cdots m_a.
\]
Thus $M_a(n)$ counts partitions of $n$ into exactly $a$ distinct part sizes, weighted by the product of their
multiplicities. Craig et al.\ \cite{CvIO} proved, in particular, that
\[
 (n^2-3n+2)M_1(n)-8M_2(n)\ge0
\]
for every positive integer $n$, with equality for $n\ge2$ if and only if $n$ is prime.
They also constructed infinitely many prime-detecting equations involving generalized MacMahon partition functions.
The author, together with Matsusaka and Shin \cite{KMS}, obtained further prime-detecting expressions for variants of MacMahon's partition functions.
These prime-detecting expressions arise naturally by expressing the corresponding quasi-modular generating functions in terms of Eisenstein series and their derivatives and applying Leli\`evre's divisor-sum criteria.

Here we construct a prime-detecting expression using normalized twisted crank moments.
We call a series
\[
 F(q)=\sum_{n\ge0}f(n)q^n
\]
prime-detecting if, for every $n\ge2$,
\[
 f(n)=0\quad\Longleftrightarrow\quad n\text{ is prime}.
\]
We impose no sign condition on the coefficients at composite indices.
The following theorem presents such an expression in $\mathscr C^{(2)}$.

\begin{theorem}\label{thm:prime-detection}
Define
\begin{equation}\label{eq:prime-detecting-form}
 \begin{aligned}
 \mathcal H(q):={}&
 (-3D^3+11D^2-16D+8)\mathbf{C}_2(-1;q)\\
 &+(-2D^2+14D-18)
       \bigl(\mathbf{C}_2(-1;q)^2\bigr)
 +16\mathbf{C}_2(-1;q)^3\\
 ={}&\sum_{n\ge1}h(n)q^n.
 \end{aligned}
\end{equation}
Then $\mathcal H\in \mathscr C^{(2)}$, and for every $n\ge2$,
\[
 h(n)
 \begin{cases}
  =0,&\text{if $n$ is prime},\\
  >0,&\text{if $n$ is composite and $4\nmid n$},\\
  <0,&\text{if $4\mid n$}.
 \end{cases}
\]
Moreover, $\mathcal H$ has highest weight eight, and no element of $\mathscr C^{(2)}$ of highest weight less than eight is prime-detecting.
\end{theorem}


\section{The algebra of twisted crank moments}
\label{sec:structure}

We recall the complete Bell polynomials and their inverse relation, following the approach in \cite{KKL}.
The complete Bell polynomials $B_n(X_1,\ldots,X_n)\in\mathbb Z[X_1,\ldots,X_n]$ are
defined by
\[
 \exp\left(\sum_{j\ge1}X_j\frac{t^j}{j!}\right)
 =\sum_{n\ge0}B_n(X_1,\ldots,X_n)\frac{t^n}{n!},
\]
where $B_0=1$.
Let \[
 Y_n=B_n(X_1,\ldots,X_n).
\]
Taking logarithms, we find
\[
 X_n
 =n![t^n]\log\left(
  1+\sum_{j\ge1}Y_j\frac{t^j}{j!}\right),
\]
where $[t^n]$ extracts the coefficient of $t^n$.
The polynomials expressing $Y_n$ in terms of $X_1,\ldots,X_n$ and $X_n$ in terms of
$Y_1,\ldots,Y_n$ have integer coefficients, as follows from their combinatorial expressions
in terms of set partitions; see \cite{KKL} for details.
Thus, for every $N\ge1$,
\[
 \mathbb C[Y_1,\ldots,Y_N]
 =\mathbb C[X_1,\ldots,X_N].
\]

We recall Rhoades's formula for the twisted crank moments \cite[Theorem~3.1]{Rhoades} in a form suited to Bell polynomial inversion.

Since
\[
 \left.\frac{\partial^r}{\partial t^r}
 C(\zeta e^t;q)\right|_{t=0}
 =C_r(\zeta;q),
\]
Taylor expansion gives
\[
 \frac{C(\zeta e^t;q)}{C(\zeta;q)}
 =\sum_{r\ge0}\mathbf C_r(\zeta;q)\frac{t^r}{r!}.
\]
Using the product formula \eqref{eq:intro-crank-gf} and the expansion of $\log(1-x)$, we obtain
\begin{equation}
 \log\frac{C(\zeta e^t;q)}{C(\zeta;q)}
 =\sum_{d,m\ge1}
 \frac{\zeta^d(e^{dt}-1)+\zeta^{-d}(e^{-dt}-1)}{d}
 q^{dm}.
\label{logexp}
\end{equation}
At $\zeta=-1$, this becomes
\[
 \begin{aligned}
 \log\frac{C(-e^t;q)}{C(-1;q)}
 &=\sum_{d,m\ge1}
   \frac{(-1)^d}{d}
   \bigl(e^{dt}+e^{-dt}-2\bigr)q^{dm}\\
 &=2\sum_{k\ge1}
   \left(\sum_{d,m\ge1}(-1)^d d^{2k-1}q^{dm}\right)
   \frac{t^{2k}}{(2k)!}.
 \end{aligned}
\]

Following Rhoades, write
\[
 \Phi_j(\tau):=\sum_{n\ge1}\sigma_j(n)q^n
\]
and define
\begin{equation}\label{eq:Fdef}
 F_{2k}(\tau)
 :=2^{2k}\Phi_{2k-1}(2\tau)-\Phi_{2k-1}(\tau),
 \qquad k\ge1.
\end{equation}
Separating the even and odd values of $d$, we obtain
\[
 \sum_{d,m\ge1}(-1)^d d^{2k-1}q^{dm}
 =2^{2k}\Phi_{2k-1}(2\tau)-\Phi_{2k-1}(\tau)
 =F_{2k}(\tau).
\]
Hence
\begin{equation}\label{eq:twistedlog}
 \log\frac{C(-e^t;q)}{C(-1;q)}
 =2\sum_{k\ge1}F_{2k}(\tau)\frac{t^{2k}}{(2k)!}.
\end{equation}
Equivalently,
\[
 \sum_{k\ge0}\mathbf C_{2k}(-1;q)\frac{t^{2k}}{(2k)!}
 =\exp\left(
  2\sum_{k\ge1}F_{2k}(\tau)\frac{t^{2k}}{(2k)!}
 \right).
\]
Comparison with the defining generating function of the complete Bell polynomials yields
\begin{equation}
 \mathbf C_{2k}(-1;q)
 =B_{2k}(0,2F_2,0,2F_4,\ldots,0,2F_{2k}).
\label{eq:belltwistedcrank}
\end{equation}
By the equality of algebras established above, we obtain
\begin{equation}\label{eq:momentcumulantalgebra}
 \mathscr C^{(2)}
 =\mathbb C[\mathbf C_{2k}(-1;q):k\ge1]
 =\mathbb C[F_{2k}:k\ge1].
\end{equation}

In particular, 
\begin{equation}\label{eq:inversebell}
 \begin{aligned}
 2F_2={}&\mathbf C_2(-1;q),\\
 2F_4={}&\mathbf C_4(-1;q)-3\mathbf C_2(-1;q)^2,\\
 2F_6={}&\mathbf C_6(-1;q)
       -15\mathbf C_2(-1;q)\mathbf C_4(-1;q)
       +30\mathbf C_2(-1;q)^3.
 \end{aligned}
\end{equation}

Inserting the Fourier expansion of $E_{2k}$ into
\eqref{eq:Fdef}, we find
\[
 F_{2k}(\tau)
 =\frac{B_{2k}}{4k}
 \bigl(2^{2k}-1-2^{2k}E_{2k}(2\tau)+E_{2k}(\tau)\bigr).
\]
For $k\ge2$, define
\begin{equation}\label{eq:Gdef}
 G_{2k}
 :=F_{2k}-\frac{B_{2k}}{4k}(2^{2k}-1).
\end{equation}
Then
\[
 G_{2k}
 =-\frac{B_{2k}}{4k}
 \bigl(2^{2k}E_{2k}(2\tau)-E_{2k}(\tau)\bigr)
 \in\mathcal M_{2k}(\Gamma_0(2)).
\]

\begin{lemma}\label{lem:modularcrank}
We have
\[
 \mathbb C[G_{2k}:k\ge2]
 =\mathbb C+
 (4E_{2,2}^2-E_4)\mathbb C[E_{2,2},E_4].
\]
\end{lemma}

\begin{proof}
For a modular form $f$ of weight $2k$ on $\Gamma_0(2)$, define the normalized Fricke action by
\[
 \mathcal W_2(f)(\tau)
 :=2^{-k}\tau^{-2k}
 f\left(-\frac1{2\tau}\right),
\]
and extend it linearly to
$\mathcal M(\Gamma_0(2))$.
The Fricke action preserves
$\mathcal M(\Gamma_0(2))$ and satisfies
\[
 \mathcal W_2(fg)=\mathcal W_2(f)\mathcal W_2(g),
 \qquad
 \mathcal W_2^2(f)=f.
\]
Thus it is an algebra automorphism of $\mathcal M(\Gamma_0(2))$ whose inverse is itself.
On the generators, it acts by
\[
 \mathcal W_2(E_{2,2})=-E_{2,2},
 \qquad
 \mathcal W_2(E_4)=5E_{2,2}^2-E_4.
\]

For $k\ge 2$, the transformation law of the Eisenstein series implies
\begin{align}
 \mathcal W_2(G_{2k})
 &=-\frac{2^kB_{2k}}{4k}
   \bigl(E_{2k}(\tau)-E_{2k}(2\tau)\bigr)\notag\\
 &=2^k
   \bigl(\Phi_{2k-1}(\tau)-\Phi_{2k-1}(2\tau)\bigr)\notag\\
 &=2^kq+O(q^2).
 \label{eq:FrickeG}
\end{align}
Thus $\mathcal W_2(G_{2k})$ has zero constant Fourier coefficient.
Every homogeneous polynomial $P$ of weight $2k$ in
$E_{2,2}$ and $E_4$ can be written uniquely as
\[
 P=(E_4-E_{2,2}^2)Q+cE_{2,2}^k,
\]
where $Q\in\mathbb C[E_{2,2},E_4]$ and $c\in\mathbb C$, by polynomial division with respect to $E_4$.
Since $E_4-E_{2,2}^2$ has zero constant Fourier coefficient and $E_{2,2}$ has constant Fourier
coefficient $1$, the constant Fourier coefficient of $P$ is $c$.
Thus $P$ has zero constant Fourier coefficient if and only if it is divisible by $E_4-E_{2,2}^2$.

It follows that $\mathcal W_2(G_{2k})$ is divisible by $E_4-E_{2,2}^2$.
Since
\begin{equation}
 \mathcal W_2(E_4-E_{2,2}^2)=4E_{2,2}^2-E_4,
\label{eq:w2e4e22}
\end{equation}
applying $\mathcal W_2$ again shows that $G_{2k}$ is divisible by $4E_{2,2}^2-E_4$.
Since
\[
 \mathbb C[E_{2,2},E_4]
 =\mathbb C[E_{2,2},4E_{2,2}^2-E_4],
\]
we may therefore write
\[
 G_{2k}
 =\sum_{j=1}^{\lfloor k/2\rfloor}
 a_{k,j}E_{2,2}^{k-2j}(4E_{2,2}^2-E_4)^j.
\]
Applying $\mathcal W_2$ and using $\mathcal W_2(E_{2,2})=-E_{2,2}$, \eqref{eq:w2e4e22}, and
\[
 E_{2,2}=1+O(q),
 \qquad
 E_4-E_{2,2}^2=192q+O(q^2),
\]
we obtain
\[
 \mathcal W_2(G_{2k})
 =192(-1)^{k-2}a_{k,1}q+O(q^2).
\]
Comparison with \eqref{eq:FrickeG} yields
\[
 a_{k,1}
 =\frac{(-1)^{k-2}2^k}{192}\ne0.
\]

We now prove by induction on $k\ge2$ that
\[
 (4E_{2,2}^2-E_4)E_{2,2}^{k-2}
 \in\mathbb C[G_{2j}:j\ge2].
\]
For $k=2$, this follows because $G_4$ is a nonzero multiple of $4E_{2,2}^2-E_4$.
For $2\le j\le\lfloor k/2\rfloor$, write
\[
 E_{2,2}^{k-2j}(4E_{2,2}^2-E_4)^j
 =(4E_{2,2}^2-E_4)^{j-1}
  \bigl((4E_{2,2}^2-E_4)E_{2,2}^{k-2j}\bigr).
\]
The first factor belongs to the generated algebra by the case $k=2$, and the second by the induction
hypothesis.
Thus every term with $j\ge2$ in the expansion of $G_{2k}$ belongs to this algebra.
Since $a_{k,1}\ne0$, the term $(4E_{2,2}^2-E_4)E_{2,2}^{k-2}$ also belongs to it,
completing the induction.

Taking products of these elements, we obtain every monomial
\[
 E_{2,2}^{a}(4E_{2,2}^2-E_4)^b,
 \qquad a\ge0,\quad b\ge1.
\]
For any $r,s\ge0$, we have
\[
 \begin{aligned}
 &(4E_{2,2}^2-E_4)E_{2,2}^rE_4^s\\
 &\quad
 =(4E_{2,2}^2-E_4)E_{2,2}^r
 \bigl(4E_{2,2}^2-(4E_{2,2}^2-E_4)\bigr)^s\\
 &\quad
 =\sum_{j=0}^s(-1)^j\binom sj4^{s-j}
 (4E_{2,2}^2-E_4)^{j+1}E_{2,2}^{r+2s-2j}.
 \end{aligned}
\]
Each summand belongs to $\mathbb C[G_{2k}:k\ge2]$ by the preceding induction.
Hence
\[
 \mathbb C+
 (4E_{2,2}^2-E_4)\mathbb C[E_{2,2},E_4]
 \subseteq\mathbb C[G_{2k}:k\ge2].
\]
The reverse inclusion follows from the divisibility of every $G_{2k}$ by $4E_{2,2}^2-E_4$.
This completes the proof.

\end{proof}

To prove Theorem~\ref{thm:structure}, we first record the modular identities
\begin{equation}\label{eq:E4E6dilations}
 \begin{aligned}
 E_4(2\tau)
 &=\frac{5E_{2,2}^2-E_4}{4},\\
 E_6(\tau)
 &=4E_{2,2}^3-3E_{2,2}E_4,\\
 E_6(2\tau)
 &=\frac{11E_{2,2}^3-3E_{2,2}E_4}{8}.
 \end{aligned}
\end{equation}
These follow from $\mathcal M(\Gamma_0(2))=\mathbb C[E_{2,2},E_4]$ by comparing the constant and $q$ coefficients in weights four and six.

Setting $\zeta=1$ in \eqref{logexp} and comparing the coefficients of $t^2$, we obtain
\begin{equation}\label{eq:c1}
 \mathbf C_2(1;q)=2\Phi_1(\tau)=\frac{1-E_2}{12}.
\end{equation}
Moreover, \eqref{eq:Fdef} and \eqref{eq:E4E6dilations}, together with the Fourier expansions of the Eisenstein series, give
\begin{equation}\label{eq:lowL}
 \begin{aligned}
 2F_2=\mathbf C_2(-1;q)
 &=\frac{3-E_2-2E_{2,2}}{12},\\
 F_4
 &=\frac{4E_{2,2}^2-E_4-3}{48},\\
 F_6
 &=\frac{3-E_{2,2}(4E_{2,2}^2-E_4)}{24}.
 \end{aligned}
\end{equation}

\begin{proof}[Proof of Theorem~\ref{thm:structure}]
By \eqref{eq:momentcumulantalgebra},
\[
 \mathscr C^{(2)}=\mathbb C[F_{2k}:k\ge1].
\]
Since
\[
 F_2=\frac{3-E_2-2E_{2,2}}{24}
\]
and $F_{2k}-G_{2k}$ is constant for every $k\ge2$, we have
\[
 \mathscr C^{(2)}
 =\mathbb C\bigl[
 E_2+2E_{2,2},\,G_{2k}:k\ge2
 \bigr].
\]

Applying Lemma~\ref{lem:modularcrank} and using
\[
 \widetilde{\mathcal M}(\Gamma_0(2))
 =\mathbb C[E_{2,2},E_4,E_2+2E_{2,2}],
\]
we obtain
\[
 \mathscr C^{(2)}
 =\mathbb C[E_2+2E_{2,2}]
 +(4E_{2,2}^2-E_4)
   \widetilde{\mathcal M}(\Gamma_0(2)).
\]

This description also shows that $\mathscr C^{(2)}$ is graded by weight, since the polynomial generator and the ideal generator are homogeneous.

To prove properness, consider the quotient map
\[
 \pi:\widetilde{\mathcal M}(\Gamma_0(2))
 \longrightarrow
 \frac{\widetilde{\mathcal M}(\Gamma_0(2))}
 {(4E_{2,2}^2-E_4)\widetilde{\mathcal M}(\Gamma_0(2))}.
\]
Since $E_2,E_{2,2},E_4$ are algebraically independent,
this quotient is naturally isomorphic to
$\mathbb C[E_2,E_{2,2}]$.
Under this identification, the image of $\mathscr C^{(2)}$ is
\[
 \pi(\mathscr C^{(2)})
 =\mathbb C[E_2+2E_{2,2}]
 \subsetneq\mathbb C[E_2,E_{2,2}].
\]
Indeed, $E_{2,2}$ cannot be expressed as a polynomial in $E_2+2E_{2,2}$, by the algebraic independence of $E_2$ and $E_{2,2}$.
Therefore
\[
 \mathscr C^{(2)}
 \subsetneq\widetilde{\mathcal M}(\Gamma_0(2)).
\]

Next, \eqref{eq:Ramanujan} and \eqref{eq:E4E6dilations} give
\begin{equation}\label{eq:level2derivatives}
 \begin{aligned}
 DE_{2,2}
 &=\frac{E_2E_{2,2}+E_4-2E_{2,2}^2}{6},\\
 DE_4
 &=\frac{E_2E_4+3E_{2,2}E_4-4E_{2,2}^3}{3},\\
 DE_2
 &=\frac{E_2^2-E_4}{12}.
 \end{aligned}
\end{equation}
Consequently,
\begin{equation}\label{eq:twistedderivatives}
 \begin{aligned}
 D(E_2+2E_{2,2})
 &=\frac{(E_2+2E_{2,2})^2
          -3(4E_{2,2}^2-E_4)}{12},\\
 D(4E_{2,2}^2-E_4)
 &=\frac{E_2-E_{2,2}}{3}(4E_{2,2}^2-E_4).
 \end{aligned}
\end{equation}
Since
$D\bigl(\widetilde{\mathcal M}(\Gamma_0(2))\bigr)
 \subseteq\widetilde{\mathcal M}(\Gamma_0(2))$,
the ideal
\[
 (4E_{2,2}^2-E_4)
 \widetilde{\mathcal M}(\Gamma_0(2))
\]
is stable under $D$.
Moreover, the derivative of a polynomial in $E_2+2E_{2,2}$ belongs to $\mathscr C^{(2)}$.
This proves
\[
 D(\mathscr C^{(2)})\subseteq\mathscr C^{(2)}.
\]

Finally, \eqref{eq:c1}, \eqref{eq:lowL} and \eqref{eq:inversebell} yield
\begin{equation}\label{eq:fullinverse}
 \begin{aligned}
 E_2
 &=1-12\mathbf C_2(1;q),\\
 E_{2,2}
 &=1+6\mathbf C_2(1;q)-6\mathbf C_2(-1;q),\\
 E_4
 &=4\bigl(1+6\mathbf C_2(1;q)
            -6\mathbf C_2(-1;q)\bigr)^2\\
 &\qquad
   -3-24\mathbf C_4(-1;q)
   +72\mathbf C_2(-1;q)^2.
 \end{aligned}
\end{equation}
Therefore
\[
 \widetilde{\mathcal M}(\Gamma_0(2))
 \subseteq
 \mathbb C[
 \mathbf C_2(1;q),
 \mathbf C_2(-1;q),
 \mathbf C_4(-1;q)].
\]
The reverse inclusion follows from quasi-modularity, completing the proof.
\end{proof}


\section{Congruences for normalized twisted moments}
\label{sec:congruences}

As noted in Section~\ref{sec:moments}, the normalized twisted moments $\mathbf C_{2k}(-1;q)$ have integer Fourier coefficients.
We begin with an elementary congruence between normalized twisted moments of different orders.
\begin{lemma}\label{lem:momentcongruence}
Let $p$ be an odd prime and let $a\ge1$.
Suppose that $r,s$ are positive even integers satisfying
\[
 r,s\ge a,
 \qquad
 r\equiv s\pmod{p^{a-1}(p-1)}.
\]
Then
\[
 \mathbf C_r(-1;q)
 \equiv\mathbf C_s(-1;q)\pmod{p^a}.
\]
In particular, for every $k\ge1$,
\[
 \mathbf C_{2k+p-1}(-1;q)
 \equiv\mathbf C_{2k}(-1;q)\pmod p.
\]
\end{lemma}

\begin{proof}
If $p\nmid m$, Euler's theorem implies
\[
 m^r\equiv m^s\pmod{p^a}.
\]
If $p\mid m$, both powers vanish modulo $p^a$ because $r,s\ge a$.
Multiplying by $(-1)^mM(m,n)$ and summing over $m$,
we obtain
\[
 C_r(-1;q)\equiv C_s(-1;q)\pmod{p^a}.
\]
Since $C(-1;q)^{-1}\in\mathbb Z[[q]]$, division by the common normalization factor preserves this congruence.
\end{proof}

We now derive the differential identities needed to prove Proposition~\ref{prop:crank-congruences}.
Using \eqref{eq:lowL} and the first identity in \eqref{eq:twistedderivatives}, we obtain
\[
 D\mathbf C_2(-1;q)
 =-\frac{(E_2+2E_{2,2})^2
          -3(4E_{2,2}^2-E_4)}{144}.
\]
Hence
\begin{equation}\label{eq:L4secondmoment}
 (2D-1)\mathbf C_2(-1;q)
       +2\mathbf C_2(-1;q)^2 =\frac{3(4E_{2,2}^2-E_4)-9}{72} =2F_4.
\end{equation}
Combining this with the second identity in \eqref{eq:inversebell}  yields
\begin{equation}\label{eq:C4diff}
 \mathbf C_4(-1;q)
 =(2D-1)\mathbf C_2(-1;q)
   +5\mathbf C_2(-1;q)^2.
\end{equation}

\begin{proof}[Proof of Proposition~\ref{prop:crank-congruences}]
Reducing \eqref{eq:C4diff} modulo $5$ gives
\[
 \mathbf C_4(-1;q)
 \equiv(2D-1)\mathbf C_2(-1;q)\pmod5.
\]
In terms of the coefficients defined by
\[
 \mathbf C_{2k}(-1;q)
 =\sum_{n\ge0}c_{2k}(n)q^n,
\]
this becomes
\[
 c_4(n)\equiv(2n-1)c_2(n)\pmod5.
\]
Taking $n=5m+3$, we obtain
\[
 c_4(5m+3)\equiv0\pmod5
 \qquad(m\ge0).
\]
Moreover, by Lemma~\ref{lem:momentcongruence},
\[
 \mathbf C_{4j}(-1;q)
 \equiv\mathbf C_4(-1;q)\pmod5
 \qquad(j\ge1).
\]
Hence
\[
 c_{4j}(5n+3)\equiv0\pmod5
 \qquad(j\ge1,\ n\ge0),
\]
as required.
\end{proof}

A further identity follows from \eqref{eq:lowL} and \eqref{eq:twistedderivatives}:
\[
 2F_6
 =4(D-1)F_4+\mathbf C_2(-1;q)
   +16\mathbf C_2(-1;q)F_4.
\]
Substituting \eqref{eq:L4secondmoment} and using the product rule yields
\begin{equation}\label{eq:L6secondmoment}
 \begin{aligned}
 2F_6={}&
 (4D^2-6D+3)\mathbf C_2(-1;q)\\
 &+(12D-12)\bigl(\mathbf C_2(-1;q)^2\bigr)
   +16\mathbf C_2(-1;q)^3.
 \end{aligned}
\end{equation}
The preceding identities and Lemma~\ref{lem:momentcongruence} yield the following congruences.

\begin{corollary}\label{cor:explicitcongruences}
The following congruences hold in $\mathbb Z[[q]]$:
\begin{equation}\label{eq:mod3}
 (D-1)\mathbf C_2(-1;q)
 +\mathbf C_2(-1;q)^2
 \equiv0\pmod3,
\end{equation}
and
\begin{equation}\label{eq:mod5}
 \begin{aligned}
 &(4D^2-6D+2)\mathbf C_2(-1;q)\\
 &\quad +(12D-12)\bigl(\mathbf C_2(-1;q)^2\bigr)
       +16\mathbf C_2(-1;q)^3
 \equiv0\pmod5.
 \end{aligned}
\end{equation}
In particular, for every $n\ge1$,
\[
 \sum_{j=1}^{n-1}c_2(j)c_2(n-j)
 \equiv(1-n)c_2(n)\pmod3.
\]
\end{corollary}

\begin{proof}
By Lemma~\ref{lem:momentcongruence}, we have
\[
 \mathbf C_4(-1;q)
 \equiv\mathbf C_2(-1;q)\pmod3.
\]
Substituting \eqref{eq:C4diff} and cancelling the invertible factor $2$ modulo $3$ proves \eqref{eq:mod3}.
The convolution congruence follows by comparing coefficients.

The same lemma also implies
\[
 \mathbf C_6(-1;q)
 \equiv\mathbf C_2(-1;q)\pmod5.
\]
Since
\[
 B_6(0,x_2,0,x_4,0,x_6)
 =15x_2^3+15x_2x_4+x_6,
\]
equation \eqref{eq:belltwistedcrank} takes the form
\[
 \mathbf C_6(-1;q)
 =2F_6+60F_2F_4+120F_2^3.
\]
Since $F_{2k}\in\mathbb Z[[q]]$ by its defining
divisor-sum expansion, it follows that
\[
 2F_6\equiv\mathbf C_2(-1;q)\pmod5.
\]
Substitution into \eqref{eq:L6secondmoment} proves
\eqref{eq:mod5}.
\end{proof}

\begin{remark}
The differential identities used here, and hence the congruences above, can also be derived from the recurrence of \cite[Theorem~6.1]{AAT}, together with Lemma~\ref{lem:momentcongruence}.
We explain this connection in Section~\ref{sec:comparison}.
\end{remark}


\section{Prime detection by twisted crank moments}
\label{sec:primes}

In this section, we prove Theorem~\ref{thm:prime-detection}.
We first express $\mathcal H$ in terms of the Eisenstein-type series $F_{2k}$ and their derivatives
and determine the signs of its Fourier coefficients.
We then prove that its highest weight is minimal among prime-detecting elements of $\mathscr C^{(2)}$.


\subsection{A divisor-sum expression and the case of odd \(n\)}

Define
\begin{align}
 U&:=2F_4-2(D^2-D+1)F_2,\label{eq:U}\\
 Z&:=2F_6-2(D^2+1)F_4+2D^2F_2.\label{eq:Z}
\end{align}
Using \eqref{eq:L4secondmoment}, \eqref{eq:L6secondmoment}, and $2F_2=\mathbf C_2(-1;q)$, we rewrite \eqref{eq:prime-detecting-form} as
\begin{equation}\label{eq:Hcumulants}
 \mathcal H
 =2F_6-2(D^2-D+3)F_4
   +2(-D^3+4D^2-3D+2)F_2.
\end{equation}
Equivalently,
\begin{equation}\label{eq:Hdef}
 \mathcal H=(D-2)U+Z.
\end{equation}
Since $F_{2k}\in\mathscr C^{(2)}$ and $\mathscr C^{(2)}$ is closed under $D$, we have
$\mathcal H\in\mathscr C^{(2)}$.
Moreover, \eqref{eq:Hcumulants} shows that its highest weight is at most eight.

\begin{remark}
The choice of $U$ is motivated by a divisor-sum identity of Leli\`evre type. As we show below, for odd $n\ge3$, the coefficient $[q^n]U$ vanishes precisely when $n$ is prime and is positive otherwise.
Applying $D-2$ preserves this property at odd indices and makes the coefficient at $q^2$ vanish.
However, $(D-2)U$ also has zero coefficient at $q^6$.
The term $Z$ corrects this failure: it preserves vanishing at every prime, while the estimates below
show that $(D-2)U+Z$ has nonzero coefficients at every composite index.
\end{remark}

Write
\[
 U=\sum_{n\ge1}u(n)q^n,
 \qquad
 Z=\sum_{n\ge1}z(n)q^n.
\]
Then
\begin{equation}\label{eq:hcoeff}
 h(n)=(n-2)u(n)+z(n).
\end{equation}
By \eqref{eq:Fdef}, the Fourier coefficients of $F_{2k}$ are
\begin{equation}\label{eq:Fcoeff}
 [q^n]F_{2k}
 =2^{2k}\sigma_{2k-1}(n/2)-\sigma_{2k-1}(n),
 \qquad n\ge1,
\end{equation}
where $\sigma_j(n/2)=0$ when $n$ is odd.

For odd $n$, it follows from \eqref{eq:U}, \eqref{eq:Z}, and \eqref{eq:Fcoeff} that
\begin{align}
 u(n)
 &=2\bigl((n^2-n+1)\sigma_1(n)-\sigma_3(n)\bigr),
 \label{eq:oddU}\\
 z(n)
 &=2\sum_{d\mid n}d(d^2-1)(n^2-d^2).
 \label{eq:oddZ}
\end{align}
We use the divisor-sum identity
\begin{equation}\label{eq:basepositive}
 (n^2-n+1)\sigma_1(n)-\sigma_3(n)
 =\sum_{d\mid n}
 d(d^2-1)\bigl((n/d)^2-1\bigr).
\end{equation}
Indeed, expanding the right-hand side and using
$\sum_{d\mid n}(n/d)=\sigma_1(n)$ gives the left-hand side.

Every summand in \eqref{eq:basepositive} is nonnegative.
It is positive precisely when $1<d<n$.
Thus $u(n)=0$ at odd primes and $u(n)>0$ at odd composite integers.
The same conclusion holds for $z(n)$ by \eqref{eq:oddZ}.
Since $n-2>0$ for odd $n\ge3$, equation \eqref{eq:hcoeff} proves
\[
 h(n)
 \begin{cases}
  =0,&\text{if $n$ is an odd prime},\\
  >0,&\text{if $n$ is odd and composite}.
 \end{cases}
\]


\subsection{The case \(4\mid n\)}

For an even integer $n$, write
\[
 n=tm,\qquad t=2^a,\quad a\ge1,\quad m\text{ odd}.
\]
By \eqref{eq:Fcoeff} and the multiplicativity of divisor sums,
\[
 \begin{aligned}
 [q^n]F_2&=(2t-3)\sigma_1(m),\\
 [q^n]F_4&=\frac{8t^3-15}{7}\sigma_3(m),\\
 [q^n]F_6&=\frac{32t^5-63}{31}\sigma_5(m).
 \end{aligned}
\]
If we put
\[
 A_1(t):=2t-3,\qquad
 A_3(t):=\frac{8t^3-15}{7},\qquad
 A_5(t):=\frac{32t^5-63}{31},
\]
then by \eqref{eq:U} and \eqref{eq:Z}, we have
\begin{align}
 \frac{u(n)}2
 &=A_3(t)\sigma_3(m)
   -(t^2m^2-tm+1)A_1(t)\sigma_1(m),
 \label{eq:evenU}\\
 \frac{z(n)}2
 &=A_5(t)\sigma_5(m)
   -(t^2m^2+1)A_3(t)\sigma_3(m)
   +t^2m^2A_1(t)\sigma_1(m).
 \label{eq:evenZ}
\end{align}

Suppose that $4\mid n$, so that $t\ge4$.
Since
\[
 \sigma_5(m)\le m^2\sigma_3(m),
 \qquad
 \sigma_1(m)\le\sigma_3(m),
\]
equation \eqref{eq:evenZ} implies
\[
 \frac{z(n)}2
 \le\bigl(m^2K(t)-A_3(t)\bigr)\sigma_3(m),
\]
where
\[
 \begin{aligned}
 K(t)
 &:=A_5(t)-t^2A_3(t)+t^2A_1(t)\\
 &=-\frac{24}{217}t^5+2t^3
   -\frac67t^2-\frac{63}{31}.
 \end{aligned}
\]
We have $K(4)=-1$.
For $t\ge8$, writing the first two terms as $t^3(2-24t^2/217)$ shows that $K(t)<0$.
Since $t$ is a power of two, it follows that 
\begin{equation}\label{eq:znegative}
 z(n)<0\qquad(4\mid n).
\end{equation}

To estimate $u(n)$, use $\sigma_3(m)\le m^2\sigma_1(m)$.
For $t\ge8$, we have
\[
 \frac{t^2m^2-tm+1}{m^2}
 \ge t^2-t+1
 >\frac{A_3(t)}{A_1(t)}.
\]
The first inequality follows from
\[
 t^2-\frac tm+\frac1{m^2}-(t^2-t+1)
 =\left(1-\frac1m\right)
   \left(t-1-\frac1m\right)\ge0,
\]
and the second from
\[
 \begin{aligned}
 &7(2t-3)(t^2-t+1)-(8t^3-15)\\
 &\qquad=(t-1)(6t^2-29t+6)>0.
 \end{aligned}
\]
Applying \eqref{eq:evenU}, we obtain
\[
 \frac{u(n)}2
 \le
 \bigl(m^2A_3(t)
 -(t^2m^2-tm+1)A_1(t)\bigr)\sigma_1(m)
 <0.
\]

If $t=4$ and $m\ge3$, then
\[
 16-\frac4m+\frac1{m^2}
 \ge\frac{133}{9}
 >\frac{71}{5}
 =\frac{A_3(4)}{A_1(4)},
\]
so the same argument shows that $u(n)<0$.
Together with \eqref{eq:hcoeff} and \eqref{eq:znegative}, this proves $h(n)<0$
for every multiple of four greater than four.
The only remaining case is $t=4$ and $m=1$, namely $n=4$.
Evaluating \eqref{eq:evenU} and \eqref{eq:evenZ} at $n=4$, we find
\[
 u(4)=12,\qquad z(4)=-144,
\]
and hence $h(4)=-120$. 
Thus $h(n)<0$ whenever $4\mid n$.


\subsection{The case \(n\equiv2\pmod4\)}

We now consider $n=2m$, where $m$ is odd.

At the prime $2$, we compute from \eqref{eq:Fcoeff}, \eqref{eq:U}, and \eqref{eq:Z} that
\[
 u(2)=8,\qquad z(2)=0.
\]
Thus $h(2)=0$ by \eqref{eq:hcoeff}.
It remains to prove that $h(2m)>0$ for odd $m\ge3$.
Since
\[
 A_1(2)=1,\qquad A_3(2)=7,\qquad A_5(2)=31,
\]
equations \eqref{eq:evenU} and \eqref{eq:evenZ} simplify to
\begin{align}
 u(2m)
 &=14\sigma_3(m)
   -2(4m^2-2m+1)\sigma_1(m),
 \label{eq:2mU}\\
 \frac{z(2m)}2
 &=31\sigma_5(m)
   -7(4m^2+1)\sigma_3(m)
   +4m^2\sigma_1(m).
 \label{eq:2mZ}
\end{align}
We will show that
\[
 h(2m)=(2m-2)u(2m)+z(2m)>0
\]
by bounding the two terms on the right separately.

We first obtain a lower bound for $z(2m)$.
We use
\[
 \sigma_5(m)\ge m^5,\qquad \sigma_1(m)\ge m,
\]
and, since $m$ is odd,
\[
 \begin{aligned}
 \frac{\sigma_3(m)}{m^3}
 &=\sum_{d\mid m}d^{-3}
 \le\sum_{\substack{d\ge1\\d\text{ odd}}}d^{-3}\\
 &\le1+\frac1{27}
      +\int_1^\infty\frac{dx}{(2x+1)^3}
 =\frac{115}{108}<\frac{16}{15}.
 \end{aligned}
\]
Consequently, \eqref{eq:2mZ} yields
\[
 \begin{aligned}
 z(2m)
 &>2\left(
 31m^5-\frac{112}{15}(4m^2+1)m^3+4m^3
 \right)\\
 &=\frac{2m^3}{15}(17m^2-52).
 \end{aligned}
\]

We next bound the possible negative contribution of $(2m-2)u(2m)$.
Discarding the positive term involving $\sigma_3(m)$ in \eqref{eq:2mU}, we obtain
\[
 (2m-2)u(2m)
 >-2(2m-2)(4m^2-2m+1)\sigma_1(m).
\]
Since
\[
 2(2m-2)(4m^2-2m+1)
 =16m^3-24m^2+12m-4<16m^3,
\]
it follows that
\[
 (2m-2)u(2m)>-16m^3\sigma_1(m).
\]
An upper bound for $\sigma_1(m)$ therefore gives the required lower bound for this contribution.
Every proper divisor of the odd integer $m$ is odd and at most $m/3$. There are
\[
 r:=\left\lfloor\frac{m/3+1}{2}\right\rfloor
\]
positive odd integers at most $m/3$, and their sum is $r^2$.
Thus
\[
 \sigma_1(m)
 \le m+r^2
 \le m+\frac{(m+3)^2}{36},
\]
and hence
\[
 (2m-2)u(2m)
 >-16m^3\left(m+\frac{(m+3)^2}{36}\right).
\]

Combining the two lower bounds, we conclude that
\begin{equation}\label{eq:finalbound}
 \begin{aligned}
 h(2m)
 &>\frac{2m^3}{15}(17m^2-52)
   -16m^3\left(m+\frac{(m+3)^2}{36}\right)\\
 &=\frac{m^3}{45}(82m^2-840m-492).
 \end{aligned}
\end{equation}
For $m\ge11$, the quadratic factor is positive, since
\[
 82m^2-840m-492
 =190+(m-11)(82m+62)>0.
\]
Therefore $h(2m)>0$ for every odd $m\ge11$.

The remaining values are $m=3,5,7,9$.
From \eqref{eq:2mU}, \eqref{eq:2mZ}, and $h(2m)=(2m-2)u(2m)+z(2m)$, we compute
\[
 h(6)=1488,\qquad
 h(10)=22224,\qquad
 h(14)=119136,\qquad
 h(18)=282096.
\]
Thus $h(2m)>0$ for every odd $m\ge3$.
This completes the proof of the coefficient assertions in Theorem~\ref{thm:prime-detection}.


\subsection{Minimal highest weight}

We now prove that no element of $\mathscr C^{(2)}$ of highest weight less than eight is prime-detecting.
Since all weights occurring in $\mathscr C^{(2)}$ are even, it suffices to consider forms of highest weight
at most six.

By Theorem~\ref{thm:structure}, this subspace has basis
\[
 \begin{gathered}
 1,\quad E_2+2E_{2,2},\quad
 (E_2+2E_{2,2})^2,\quad 4E_{2,2}^2-E_4,\\
 (E_2+2E_{2,2})^3,\quad
 (E_2+2E_{2,2})(4E_{2,2}^2-E_4),\quad
 E_{2,2}(4E_{2,2}^2-E_4).
 \end{gathered}
\]
To examine Fourier coefficients at primes, it is more convenient to use the equivalent basis
\begin{equation}\label{eq:lowbasis}
 1,\quad F_2,\quad DF_2,\quad D^2F_2,\quad
 F_4,\quad DF_4,\quad F_6.
\end{equation}
Indeed, using \eqref{eq:lowL} and \eqref{eq:twistedderivatives}, we compute
\[
 \begin{aligned}
 DF_2
 &=-\frac{(E_2+2E_{2,2})^2
          -3(4E_{2,2}^2-E_4)}{288},\\
 DF_4
 &=\frac{(E_2-E_{2,2})(4E_{2,2}^2-E_4)}{144},\\
 D^2F_2
 &=-\frac{(E_2+2E_{2,2})^3}{1728}
   +\frac{E_2(4E_{2,2}^2-E_4)}{192}.
 \end{aligned}
\]
Together with the formulas for $F_2,F_4,$ and $F_6$ in \eqref{eq:lowL}, these identities recover all
elements of the preceding basis.

Let $f$ belong to this subspace and suppose that its coefficient at every odd prime vanishes.
Write
\[
 f=c+a_1F_2+a_2DF_2+a_3D^2F_2
      +a_4F_4+a_5DF_4+a_6F_6.
\]
For an odd prime $p$, we have by \eqref{eq:Fcoeff}
\[
 [q^p]F_{2k}=-(1+p^{2k-1}).
\]
Since $D$ multiplies the coefficient of $q^p$ by $p$, the condition $[q^p]f=0$ becomes
\[
 \begin{aligned}
 0={}&a_1(1+p)+a_2(p+p^2)+a_3(p^2+p^3)\\
    &+a_4(1+p^3)+a_5(p+p^4)+a_6(1+p^5).
 \end{aligned}
\]
This polynomial vanishes at infinitely many values of $p$, so all its coefficients are zero.
The terms of degrees five and four give $a_6=a_5=0$, and the remaining coefficients satisfy
\[
 (a_1,a_2,a_3,a_4)=\lambda(-1,1,-1,1)
\]
for some $\lambda\in\mathbb C$.
Thus vanishing at all odd primes forces
\[
 f=c+\lambda\bigl(F_4-(D^2-D+1)F_2\bigr).
\]

It remains to impose vanishing at the prime $2$.
By \eqref{eq:Fcoeff},
\[
 [q^2]F_2=1,\qquad [q^2]F_4=7,
\]
and hence
\[
 [q^2]f
 =\lambda\bigl(7-(2^2-2+1)\bigr)
 =4\lambda.
\]
Therefore $[q^2]f=0$ forces $\lambda=0$.
The only forms in this subspace whose coefficients vanish at every prime are consequently constants.
A constant is not prime-detecting, since all of its positive-index coefficients vanish.

We have proved that a prime-detecting element of $\mathscr C^{(2)}$ must have highest weight at least eight.
The form $\mathcal H$ is prime-detecting and has highest weight at most eight, so its highest weight is exactly eight.
This completes the proof of Theorem~\ref{thm:prime-detection}.

\begin{remark}
The form $\mathcal H$ lies in the span of constants,
Eisenstein series at $\tau$ and $2\tau$, and their derivatives. Indeed, \eqref{eq:Hcumulants} expresses
$\mathcal H$ as a linear combination of derivatives of $F_2,F_4,$ and $F_6$, each of which is a constant
plus a linear combination of Eisenstein series.
Thus the Eisenstein nature of our prime-detecting form follows directly from its construction.
This is consistent with the structural results on prime-detecting quasi-modular forms in
\cite{KaneKL,KwonLee}.
\end{remark}


\section{A MacMahon variant and twisted crank moments}
\label{sec:comparison}

For $k\ge1$, we consider MacMahon's function
\[
 B_k(q)
 :=\sum_{0<m_1<\cdots<m_k}
 \prod_{i=1}^k\frac{q^{m_i}}{(1+q^{m_i})^2},
 \qquad B_0(q):=1.
\]
In \cite[Corollary~3.9]{KMS}, we showed that
\[
 B_k(q)\in\widetilde{\mathcal M}(\Gamma_0(2)).
\]
We now prove that these functions generate the same algebra as the normalized twisted crank moments.

\begin{proposition}\label{prop:Bcrank}
We have
\[
 \mathbb C[B_k:k\ge1]=\mathscr C^{(2)}.
\]
\end{proposition}

\begin{proof}
By the product formula \eqref{eq:intro-crank-gf},
\[
 \frac{C(-e^t;q)}{C(-1;q)}
 =\prod_{m\ge1}
 \frac{(1+q^m)^2}
 {(1+e^tq^m)(1+e^{-t}q^m)}.
\]
Since
\[
 (1+e^tq^m)(1+e^{-t}q^m)
 =(1+q^m)^2+(2\cosh t-2)q^m,
\]
we obtain
\[
 \frac{C(-e^t;q)}{C(-1;q)}
 =\prod_{m\ge1}
 \left(
 1+(2\cosh t-2)\frac{q^m}{(1+q^m)^2}
 \right)^{-1}.
\]
On the other hand, by the definition of $B_k(q)$,
\[
 \prod_{m\ge1}
 \left(1+x\frac{q^m}{(1+q^m)^2}\right)
 =\sum_{k\ge0}B_k(q)x^k,
\]
because the coefficient of $x^k$ is obtained by choosing the nonconstant term from exactly $k$ distinct factors.
Substituting $x=2\cosh t-2$ yields \begin{equation}\label{eq:Bbridge}
 \frac{C(-e^t;q)}{C(-1;q)}
 =\left(
 \sum_{k\ge0}B_k(q)(2\cosh t-2)^k
 \right)^{-1}.
\end{equation}

To compare coefficients, first expand
\[
 \left(1+\sum_{j\ge1}B_j(q)x^j\right)^{-1}
 =1-\sum_{j\ge1}B_j(q)x^j
   +\left(\sum_{j\ge1}B_j(q)x^j\right)^2-\cdots.
\]
The coefficient of $x^k$ is $-B_k$ plus a polynomial in $B_1,\ldots,B_{k-1}$, because in every term involving two or more factors, the positive indices sum to $k$, so each is less than $k$.

Now substitute
\[
 x=2\cosh t-2=t^2+O(t^4).
\]
Then $x^k=t^{2k}+O(t^{2k+2})$, while powers $x^j$
with $j>k$ do not contribute to the coefficient of $t^{2k}$.
The contributions from $j<k$ involve only
$B_1,\ldots,B_{k-1}$.
Comparing with
\[
 \frac{C(-e^t;q)}{C(-1;q)}
 =\sum_{k\ge0}\mathbf C_{2k}(-1;q)
   \frac{t^{2k}}{(2k)!},
\]
we obtain
\[
 \frac{\mathbf C_{2k}(-1;q)}{(2k)!}
 =-B_k+P_k(B_1,\ldots,B_{k-1})
 \qquad(k\ge1)
\]
for some polynomial $P_k$ with rational coefficients.
Thus each normalized twisted moment is a polynomial in the functions $B_j$.
Conversely, starting with $\mathbf C_2(-1;q)=-2B_1$, we can solve these identities successively for $B_k$.
Each $B_k$ is a polynomial in $\mathbf C_2(-1;q),\ldots,\mathbf C_{2k}(-1;q)$,
which proves the equality of the two algebras.
\end{proof}

The functions $B_k(q)$ are denoted by $U_k(2,q)$ in Amdeberhan, Andrews, and Tauraso
\cite[Theorem~6.1]{AAT}, who proved the recurrence
\begin{equation}\label{eq:Brecurrence}
 k(2k-1)B_k
 =\left(D+B_1-\binom{k}{2}\right)B_{k-1},
 \qquad k\ge1.
\end{equation}
Equivalently,
\[
 DB_k
 =(k+1)(2k+1)B_{k+1}
 +\left(\binom{k+1}{2}-B_1\right)B_k.
\]
Thus $\mathbb C[B_k:k\ge1]$ is closed under $D$.
Proposition~\ref{prop:Bcrank} then provides another proof that $\mathscr C^{(2)}$ is closed under $D$.

The same correspondence explains the connection with the differential identities used in
Section~\ref{sec:congruences}.
For example, comparing coefficients in \eqref{eq:Bbridge},
we find
\[
 \mathbf C_2(-1;q)=-2B_1,
 \qquad
 \mathbf C_4(-1;q)=-2B_1+24B_1^2-24B_2.
\]
Substituting the case $k=2$ of \eqref{eq:Brecurrence},
namely
\[
 6B_2=DB_1+B_1^2-B_1,
\]
we recover
the differential identity \eqref{eq:C4diff}.
The identity for the sixth moment follows similarly by using the case $k=3$.
Thus the recurrence provides an alternative derivation of the differential identities underlying the
congruences proved above.

Furthermore, using the identity $\mathbf C_2(-1;q)=-2B_1(q)$,
we obtain an explicit prime-detecting expression involving only $B_1(q)$ and its products and derivatives,
addressing the problem left open at the end of Section~4 of \cite{KMS}.

\begin{corollary}\label{cor:B-prime-detection}
Define
\[
 \begin{aligned}
 \mathcal B(q):={}&
 (-3D^3+11D^2-16D+8)B_1(q)\\
 &+(4D^2-28D+36)\bigl(B_1(q)^2\bigr)
   +64B_1(q)^3\\
 =:{}&\sum_{n\ge1}b(n)q^n.
 \end{aligned}
\]
For every $n\ge2$, we have
\[
 b(n)
 \begin{cases}
 =0,&\text{if $n$ is prime},\\
 <0,&\text{if $n$ is composite and $4\nmid n$},\\
 >0,&\text{if $4\mid n$}.
 \end{cases}
\]
Moreover, $\mathcal B$ has highest weight eight, which is minimal among prime-detecting elements
of $\mathbb C[B_k:k\ge1]$.
\end{corollary}

\begin{proof}
Substituting $\mathbf C_2(-1;q)=-2B_1(q)$ into \eqref{eq:prime-detecting-form}, we obtain
$\mathcal B=-\mathcal H/2$.
The assertions follow from Theorem~\ref{thm:prime-detection} and
Proposition~\ref{prop:Bcrank}.
\end{proof}


\section*{Acknowledgments}
The author thanks Toshiki  Matsusaka for suggesting
that the resolution of the problem left open in
\cite{KMS} be stated explicitly, as in
Corollary~\ref{cor:B-prime-detection}.
This work was supported by the Basic Science Research Program through the National Research Foundation of Korea (NRF) funded by the Ministry of Education (RS-2025-25415913) and the NRF grant funded by the Korea government (MSIT) (RS-2026-25606469).


\end{document}